\documentclass[reqno,11pt]{amsart}
\usepackage{a4wide,color,eucal,enumerate,mathrsfs}
\usepackage[normalem]{ulem}
\usepackage{amsmath,amssymb,epsfig,amsthm} 
\usepackage[latin1]{inputenc}
\usepackage{psfrag}
\usepackage{hyperref,cleveref}

\def\az{\alpha}

\def\dist{{\mathop\mathrm{\,dist\,}}}
\def\loc{{\mathop\mathrm{\,loc\,}}}

\def\dz{\delta}

\def\ez{\epsilon}

\def\bint{{\ifinner\rlap{\bf\kern.35em--}
\int\else\rlap{\bf\kern.45em--}\int\fi}\ignorespaces}

\def\bbint{{\ifinner\rlap{\bf\kern.35em--}
\hspace{0.078cm}\int\else\rlap{\bf\kern.45em--}\int\fi}\ignorespaces}

\newcommand{\R}{\mathbb{R}}

\newtheorem{thm}{Theorem}[section]
\newtheorem{lem}[thm]{Lemma}
\newtheorem{prop}[thm]{Proposition}
\newtheorem{defn}[thm]{Definition}
\newtheorem{conj}[thm]{Conjecture}
\numberwithin{equation}{section}

\theoremstyle{remark}
\newtheorem{rem}[thm]{Remark}

\def\bint{{\ifinner\rlap{\bf\kern.35em--}
\int\else\rlap{\bf\kern.45em--}\int\fi}\ignorespaces}

\usepackage{graphicx}
\usepackage{import}
\usepackage{xifthen}
\usepackage{pdfpages}
\usepackage{transparent}

\title[Serrin's overdetermined theorem within Lipschitz domains]{Serrin's overdetermined theorem within Lipschitz domains}
 \author{Hongjie Dong and Yi Ru-Ya Zhang}
\date{\today}

 \address{Division of Applied Mathematics, Brown University, Providence RI 02912, USA}
  \email{Hongjie\_Dong@Brown.edu}  

\address{State Key Laboratory of Mathematical Sciences, Academy of Mathematics and Systems Science, Chinese Academy of Sciences, Beijing 100190, China}
\address{Institute of Mathematics, Academy of Mathematics and Systems Science, the Chinese Academy of Sciences, Beijing 100190, China}
 \email{yzhang@amss.ac.cn}

  \thanks{ H. Dong was partially supported by the NSF under agreement DMS-2350129. Y. R.-Y. Zhang  is partially funded by  National Key R\&D Program of China (Grant No. 2025YFA1018400 \&  No. 2021YFA1003100), NSFC grant No. 12288201 \& No. 12571128, the Chinese Academy of Sciences, and CAS Project for Young Scientists in Basic Research, Grant No. YSBR-031.}


\begin{document}
\begin{abstract}
Let $\Omega\subset\mathbb R^n$ be a Lipschitz domain. We prove that, 
$\Omega$ satisfies the following Serrin-type overdetermined system
   $$u \in W^{1,2}(\mathbb R^n), \quad  u=0\  \text{ a.e. in }\R^n\setminus \Omega,\quad \Delta  u=\mathbf{c}\mathscr{H}^{n-1}|_{\partial^*\Omega} - \mathbf{1}_{\Omega}\,dx,$$
 in the weak sense if and only if $\Omega$ is a ball. Here   $\mathscr H^{n-1}$ denotes the $(n-1)$-dimensional Hausdorff measure.  Moreover,   a generalization of our method in the anisotropic setting is discussed.  
Our approach   offers an alternative proof to \cite{FZ2025} in the case of Lipschitz domains, introducing a novel viewpoint to settle \cite[Question 7.1]{HLL2024}. 
\end{abstract}
 
\maketitle

\section{Introduction}
Let $\Omega\subset \mathbb R^n$ be a bounded  domain.  
When $\Omega$ is of class  $C^2$, it was proven in \cite{S1971,W1971} that  $\Omega$ admits a solution to the following overdetermined system:  
\begin{equation}\label{u}
    -\Delta u =1 \ \text{ in }\ \Omega, \ u=0 \ \text{ almost everywhere in} \  \mathbb R^n\setminus \Omega,\ \text{ and }\ \partial_\nu u=-\mathbf{c} \ \text{ on } \ \partial\Omega, 
\end{equation}
where $\mathbf{c}= \frac{|\Omega|}{\mathscr H^{n-1}(\partial \Omega)}$ is a constant,
if and only if $\Omega$ is a ball. Here,  $\nu$ denotes the unit outer normal of $\Omega$. Subsequent research has explored alternative proofs of this result; see, e.g.  \cite{BNST2008, CH1998, PS1989}. 
This discovery initiated a vibrant and fruitful field of study, where the interplay of analysis and geometry has yielded numerous applications across diverse areas of mathematics and the natural sciences.

\subsection{Serrin's theorem in rough domains}
In 1992, Vogel \cite{V1992} established that if $\Omega$ is a $C^1$ domain admitting a solution to \eqref{u}, then $\Omega$ is necessarily of class  $C^{2}$ and thus Serrin's theorem applies. Specifically, Vogel assumed the following boundary behavior for the solution 
$$u(x)\to 0 \quad \text{ and } \quad  |\nabla u|(x)\to \mathbf{c} \ \text{ uniformly as } x\to \partial \Omega,$$
and the regularity theory for free boundary problems of Alt-Caffarelli type to deduce the $C^2$-regularity of  $\Omega$.

Later, Berestycki raised the following question:\\
{\it Suppose $\Omega$ is $C^2$ throughout except for a potential corner, and $u$ represents a strong solution to \eqref{u} everywhere except at said corner. Does Serrin's theorem remain applicable in this scenario?}\\
This problem was resolved in \cite{P1998} via an adapted moving plane method that systematically avoided the singular point.
Subsequently, interest grew in extending Serrin's theorem to more general domains. In particular, \cite[Question 7.1]{HLL2024} posed the following:\\
{\it Does Serrin's theorem  hold if $\Omega$ is merely Lipschitz and $u$ solves the equation under the weak formulation?}

Let $B$ be the unit ball in $\mathbb R^n$.
Recall that, a measurable set $\Omega\subseteq \R^n$ has \emph{finite perimeter} if the distributional gradient of its characteristic function  $\mathbf{1}_\Omega$ is a $\R^n$-valued Radon measure $D\mathbf{1}_\Omega$with finite total variation, i.e., $|D\mathbf{1}_\Omega|(\R^n)<\infty$.
By the Lebesgue-Besicovitch differentiation theorem, for  $|D\mathbf{1}_{\Omega}|$-a.e. $x$, the following holds:
\begin{equation*} 
\lim_{r\to 0^+}\frac{D\mathbf{1}_\Omega(x+rB)}{|D\mathbf{1}_\Omega|(x+rB)} = -\nu_x \quad \text{and} \quad |\nu_x| = 1.
\end{equation*}
The set of such $x$ is called the \emph{reduced boundary} of  $\Omega$, denoted by $\partial^*\Omega$.
 At these points,  $\nu_x$  is the measure-theoretic outer unit normal to $\Omega$ at $x$.
By De Giorgi's rectifiability theorem,  the reduced boundary is a $(n-1)$-rectifiable set. Moreover, $\Omega$ can be modified on a null set so that $\overline{\partial^* \Omega}=\partial\Omega$.
Finally,  a set $E$ of finite perimeter is \emph{indecomposable} if for any   $ F \subset  E$ of finite perimeter satisfying
 $$\mathscr H^{n-1}(\partial^* E) = \mathscr H^{n-1}(\partial^* F) + \mathscr H^{n-1}(\partial^* (E\setminus F)), $$
either $|F|=0$ or $|E\setminus F|=0$. 
For further details, see \cite{ACMM2001} and \cite[Sections 12, 15]{M2012}.

In \cite{FZ2025}, the authors proved that if a bounded indecomposable set of finite perimeter $\Omega$ satisfies, for some $A>0$, the measure-theoretic condition
\begin{equation}\label{FZ measure density}
    \mathscr H^{n-1}(B_r(x)\cap \partial^* \Omega)\le A r^{n-1} \quad \text{ for $\mathscr H^{n-1}$-a.e. } x\in \partial^* \Omega \text{ and $r \in (0,1)$,}
\end{equation}  
along with the weak formulation of \eqref{u}:
\begin{equation}\label{weak}
    u \in W^{1,2}_0(\Omega)\quad\text{and}\quad \int_{\Omega} \nabla u\cdot \nabla\varphi \, dx = -\mathbf{c}\int_{\partial^* \Omega} \varphi \, d\mathscr{H}^{n-1} + \int_{\Omega} \varphi \, dx\quad \forall\, \varphi\in C^1(\mathbb{R}^n),
\end{equation}
then $\Omega$ must be a ball. 
Notably, Lipschitz domains satisfy the above measure-theoretic condition, thereby resolving \cite[Question 7.1]{HLL2024}. Furthermore, the same boundary assumption ensures that $u$ is Lipschitz up to the boundary \cite[Lemma 2.1]{FZ2025}. 

The approach outlined in \cite{FZ2025} builds upon a generalization of the argument presented in \cite{W1971}, drawing from methodologies in geometric measure theory and the one-phase Altar-Caffarelli-type free boundary problem; for further reference, see, e.g., \cite{V2023}. A pivotal step in this method involves demonstrating that the solution satisfying \eqref{weak} is Lipschitz continuous up to the boundary, a property that heavily relies on the boundedness of the Neumann data. 

In the present manuscript, we explore a more intuitive strategy to prove Serrin's theorem across all Lipschitz domains, while it is still a version of the argument by Weinberger \cite{W1971}. This is achieved by approximating the domain $\Omega$ from inside using the (open) super-level sets $\{u>\ez\}$ of $u$.  However, this task is far from straightforward: For example, it is  uncertain whether  $\mathscr H ^{n-1}(\{u>\ez\})$ converges to $\mathscr H ^{n-1}(\partial \Omega)$ given that only lower semicontinuity is initially guaranteed. 

Nevertheless, when $\Omega$ is a Lipschitz domain and $u$ is harmonic, it is well known that if $u=0$ on an open subset $\Gamma\subset \partial\Omega$, one can understand the (outer) normal derivative $\partial_\nu u$ on $\Gamma$ in the weak formulation \eqref{weak} via nontangential limits. Moreover, $\partial_\nu u\in L^{2+\dz}(\Gamma;\,d\mathscr H^{n-1})$ for some $\dz=\dz(n)>0$; see e.g. \cite{C1985}. 
This observation motivates our study of  Serrin's overdetermined problem via tools from harmonic analysis and introducing a novel approach to address \cite[Question 7.1]{HLL2024}.  

\begin{thm}\label{main thm}
    Let $\Omega$ be a Lipschitz domain. 
    Assume that $u\in W^{1,\,2}(\mathbb R^n)\cap C^2_\loc(\Omega)$ satisfies the overdetermined system \eqref{weak}. Then 
    $$u(x)=\frac{r^2-|x|^2}{2n} \quad \text{ in  }  \Omega  \text{ for some $r>0$,}$$
where $H_*$ is the dual function of $H$. Moreover, this implies that $\Omega$ must be a Euclidean ball. 
\end{thm}

Our method also stands out as it does not rely on the boundedness of the Neumann data, thereby offering valuable insights for further exploring the behavior of Neumann boundary values of torsion functions across Lipschitz domains that vary from the unit ball. 

\subsection{Wulff potential and nontangential limits}

We next introduce the \emph{anisotropic energy} framework. Let $K$ be  a convex body containing the origin, and  $H\colon \mathbb R^n\to \mathbb R$ denote the corresponding Wulff potential (or \emph{supporting function}), defined by 
$$H(x)=\sup_{y\in K} x\cdot y,$$
which is $1$-homogeneous, that is,
$$H(tx)=tH(x)\quad \text{  for every $t>0$ and $x\in \mathbb R^n$} .$$
This induces a general Minkowski norm on $\mathbb R^n$. We define the dual function $H_*$ of $H$ as
$$H_*(x)= \sup\left\{x \cdot y : H(y) = 1\right\}\quad  x\in \mathbb R^n.$$
Equivalently,  $K$ can be characterized as
$$K=\{x\in \mathbb R^n\colon H_*(x)\le 1\}. $$

Let $V(x)=\frac 1 2 H^2(x)$.  We consider the following anisotropic overdetermined system:
\begin{equation*}
     u \in W^{1,2}_0(\Omega) \quad \text{and} \quad \int_{\Omega} DV(Du)\cdot D\varphi \, dx = -\mathbf{c}\int_{\partial \Omega^*} \varphi H(-\nu)\, d\mathscr{H}^{n-1} + \int_{\Omega} \varphi \, dx, \  \forall\, \varphi\in C^1(\mathbb{R}^n),
\end{equation*}
where $\mathbf c=\frac{|\Omega|}{P_K(\Omega)}$, $\nu$ is the (measure-theoretic) unit outer normal, and $$P_K(\Omega)=\int_{\partial^*\Omega} H(-\nu)\, d\mathscr H^{n-1}$$
denotes the anisotropic perimeter of $\Omega$. Then further denoting the anisotropic Laplacian by
$$\Delta_Hu:={\rm div} (DV(Du))= {\rm div} (H(Du)DH(Du)),$$
one may equivalently restate the equation in distributional form as
\begin{equation}\label{weak wulff u}
   u \in W^{1,2}(\mathbb R^n), \quad  u=0\  \text{ a.e. in }\R^n\setminus \Omega,\quad \Delta_H u=\mathbf{c}H(-\nu)\mathscr{H}^{n-1}|_{\partial^*\Omega} - \mathbf{1}_{\Omega}\,dx. 
\end{equation} 
When $\Omega$ is sufficiently smooth, ensuring that $u\in C^2(\overline{\Omega})$, i.e., $u$ is a classical solution, it was independently demonstrated in the two exemplary work \cite{CS2009} and \cite{WX2011} that $\Omega$ must coincide with $K$ up to an appropriate translation and dilation. These proofs elegantly refined the methodologies introduced in \cite{BNST2008} and \cite{W1971}, respectively.

 One can generalize our approach for Theorem~\ref{main thm}  to establish an anisotropic variant of Serrin's overdetermined theorem for Lipschitz domains, with additional assumptions on $u$.

\begin{thm}\label{main thm 2}
    Let $\Omega$ be a Lipschitz domain with a local Lipschitz constant $L>0$, $K$ be a (bounded) uniformly convex set of $C^{3}$ class with the principle curvatures of $\partial K$ bounded from above by $\kappa$ and from below by $\kappa^{-1}>0$, 
    and $H$ be the associated Wulff potential.
    Assume that $u\in W^{1,\,2}(\mathbb R^n) $
    satisfies the anisotropic overdetermined system \eqref{weak wulff u}. Moreover, we suppose that $D^2u\in L^n$ and $|Du|\ge c_0>0$ in some neighborhood of $\partial \Omega$. Then there exists a constant $\ez_0=\ez_0(n,\,\kappa,\,\|K\|_{C^{3}})>0$ satisfying that, whenever
    $ L \le \ez_0,$
    one has
    $$u(x)=\frac{r^2-H_*^2(x)}{2n} \quad \text{ in  }  \Omega  \text{ for some $r>0$,}$$
where $H_*$ is the dual function of $H$. Moreover, this implies that $\Omega$ and $K$ are homothetic. 
\end{thm}

\begin{rem} We provide a series of remarks below to clarify the context and technical aspects of our  Theorem~\ref{main thm} and Theorem~\ref{main thm 2}.
\noindent\begin{enumerate}
\item Due to our reliance on harmonic analysis techniques applied to linear operators, we have to require some assumptions inherently involves $Du$; see Defintion~\ref{DKP defn} and Theorem~\ref{KP} below. In particular, the smallness of the local Lipschitz constant $L$ should be understood via \cite[Definition 2.2]{DPR2017}, which is valid whenever the domain is $C^1$.

It is also crucial to acknowledge that the proof in \cite{FZ2025} heavily depends on a set of properties unique to the Laplacian, so extending that approach to the anisotropic case appears to be non-trivial.

\item The core technique underlying both of our results involves proving that the non-tangential maximal function $N_*(|Du|)$ belongs to $L^{2+\delta}(\partial\Omega)$ for some $\delta>0$. This estimate is then applied to control the  $L^2$-convergence of 
$H(Du)$ to $\mathbf {c}$ as $\ez\to 0$, where $\Omega_\ez=\{u>\ez\}$. Recall that in \cite{FZ2025}, $u$ is shown to be Lipschitz continuous up to the boundary (under a weaker assumption \eqref{FZ measure density}), providing an $L^\infty$-analog of this dominance.  

While the $L^{2+\delta}$-integrability of the non-tangential maximal function of $|Du|$ is well established for harmonic functions in more general domains (including Lipschitz domains), the Lipschitz regularity assumption is crucial for analyzing restrictions of  $W^{2,\,2}$-functions; see the proof of Proposition~\ref{KP}. Indeed, (a variant of) the second order Sobolev regularity of $u$ also turns out  to be crucial in generalizing the method in \cite{FZ2025} in the forthcoming manuscript \cite{FZ2026}. 
Consequently, this limitation prevents us from further extending our results to broader classes of domains.

\item Similar to \cite{FZ2025}, our approach follows Weinberger's framework \cite{W1971}. Readers familiar with \cite{BNST2008} may wonder whether its techniques could be adapted here. Notably, the proof in \cite{BNST2008} relies solely on the minimum principle for $u$, without requiring the maximum principle for $|Du|$ or the Hopf boundary lemma. However, extending their method would necessitate $L^3$-convergence of $|Du|$ on $\partial\Omega_{\ez}$   to the boundary value $\mathbf {c}$, a condition that remains unclear even in Lipschitz domains, especially when the dimension is high.
\end{enumerate}
\end{rem}

\subsection{Alexandrov theorem and Maggi's conjecture}
Our result is also partially motivated by a conjecture of Maggi \cite{M2018}:
\begin{conj}[Maggi's Conjecture]\label{magg conj}
For a positive convex integrand, Wulff shapes are the unique sets of finite perimeter and finite volume that are critical points of the anisotropic boundary energy under fixed volume constraints.
\end{conj}
This conjecture generalizes the classical Alexandrov theorem \cite{A1958, A1962} and its anisotropic counterpart \cite{HLMG2009}. In Conjecture~\ref{magg conj}, the notions of first variation and critical points are rigorously defined via the convexity of the functional along prescribed variational flows \cite[pages 35-36]{M2018}.

In full generality (for sets of finite perimeter in dimensions greater than $2$), the conjecture has been resolved in the following cases. For the standard Euclidean norm, we refer to the manuscript of Delgadino and Maggi \cite{DM2019}. For elliptic integrands of class $C^{2,\,\az}$ ($\az>0$),  see the work by De Rosa--Kolasi\'nski--Santilli \cite{DKS2020}. The remaining cases are still open. 
Both arguments extend the approach of Montiel and Ros \cite{MR1991} by establishing the Heintze--Karcher inequality for weakly mean convex domains. The central challenge lies in correctly interpreting the weak mean curvature, which requires properties such as
$$\mathscr H^{n-1}(\overline {\partial^*\Omega}\setminus \partial^*\Omega)=0.$$
For a more detailed discussion, see, e.g., \cite[Section 4]{D2024}.

Alexandrov's theorem is closely connected to Serrin's overdetermined boundary value problem. For instance, \cite{CH1998} established Serrin's theorem via Alexandrov's result, and subsequent work \cite{CM2017, MP2019, MP2020} have further explored this deep interplay. We also recommend the insightful survey \cite{M2017}, which provides a comprehensive review of these developments.

Additionally, since Alexandrov's theorem is related to minimal surfaces, Serrin's overdetermined problem is similarly linked to the one-phase Bernoulli free boundary problem \cite{V2023}. This motivates the following anisotropic generalization of Serrin's theorem:

\begin{conj}[Anisotropic Serrin-type conjecture for rough domains]
For a positive convex integrand $H$, Wulff shapes are the unique (indecomposable) sets of finite perimeter (or at least among Lipschitz domains) and finite volume that satisfy the overdetermined system \eqref{weak wulff u}.
\end{conj}  
A key implication of this conjecture is that, the upper bound condition on the perimeter density \eqref{FZ measure density} in \cite[Theorem 1.3]{FZ2025} might be dispensed, though this remains an open problem. As previously highlighted in \cite[Remark 1.5]{FZ2025}, this problem presents greater challenges than its counterpart in the context of the Alexandrov theorem, even when the set under consideration is merely Lipschitz.

The manuscript is organized as follows. In Section 2, we present the proof of Theorem~\ref{main thm}, namely Serrin's result in the isotropic setting, following and generalizing Weinberger's original argument \cite{W1971}. In Section 3, we introduce the terminology from harmonic analysis needed to treat the anisotropic case. In particular, we show that the (nonlinear) operators $\Delta_H$ is uniformly elliptic and satisfy the DKP condition. This allows us to characterize the Neumann data of the weak solutions via nontangential limits and to obtain strong $L^{2+\delta}$ convergence of $|Du|$ (respectively, $H(Du)$) up to the boundary in  the isotropic (respectively,  anisotropic) case. These estimates constitute the key ingredient for extending Weinberger's argument.

\medskip

\noindent{\bf{Data availability statement}}: Data sharing is not applicable to this article as no new data were created or analyzed in this study.

\section{Proof of Theorem~\ref{main thm}}
Let us fix some notation. 
For a point $x\in \Omega$, we write
$$d(x)=\dist(x,\,\partial\Omega).$$
We denote by $\mathcal L^n$ the Lebesgue measure in $\mathbb R^n,$ and by $\mathscr H^{n-1}$ the $(n-1)$-dimensional Hausdorff measure. 
A general constant is represented by the symbol $C$, which may differ in various estimates. We incorporate all the parameters it relies on within parentheses, denoted as $C(\cdot)$.

Assume that $x\in \partial\Omega$ and $\partial\Omega$ locally is a graph of an $L$-Lipschitz function near $x$ for some $L>0$. We define the  vertical cone $V_L(x)$ with vertex at $x$ (of sufficiently narrow aperture) as
$$ \left\{y\in \Omega\colon \frac{y-x}{|y-x|}\cdot e_n \ge \frac{2L}{\sqrt{1+4L^2}}\right\}.  $$ 
Then the \emph{non-tangential maximal function} $N_* (|Du|)$ is defined as
$$N_*(|Du|)(x)=\sup_{y\in V_L(x)} |Du|(y).$$

\begin{lem}\label{convergence Laplacian}
        Let $\Omega$ and $u$  be  as in Theorem~\ref{main thm}. Then there exists $\delta=\delta(n)>0$ for which,  the nontangential maximal function 
    $N_*(|Du|)\in L^{2+\delta }(\partial\Omega).$
    In particular, 
    $$|Du|\bigg|_{\partial \Omega_\ez} \to \mathbf{c}$$
      strongly in $L^{2+\delta}$ with respect to $\mathscr H^{n-1}$-measure as $\ez\to 0$, where $\Omega_\ez=\{u>\ez\}$. 
\end{lem}
\begin{proof}
Let $v=\chi_{\Omega}\star G$, where $G$ is the fundamental solution of the operator $\Delta={\rm div}(D\cdot)$ in the whole space $\mathbb R^n$.   Then it follows that $u-v$ is  harmonic in $\Omega$ and 
$$(u-v)|_{\partial\Omega}=-v|_{\partial\Omega}$$ 
is the trace of a $W^{2,\,2}$-function on the boundary $\partial\Omega$ of a Lipschitz domain. 

By the trace lemma and the fractional Sobolev embeddings,
$$Dv|_{\partial\Omega}\in W^{\frac 1 2, 2}(\partial\Omega;\,d\mathscr H^{n-1})\hookrightarrow L^{2+\delta}(\partial\Omega;\,d\mathscr H^{n-1})\quad \text{ for some } \dz=\dz(n)>0,$$
$D_\tau(u-v)\in L^{2+\delta}(\partial\Omega;\,d\mathscr H^{n-1})$. 
Now we can apply \cite{V1984} (see also \cite[(1.4)]{DK1987}) and \cite[Theorem 6.17]{KP1993} to $u-v$ to conclude the lemma. 
\end{proof}

As $u$ is $C^2_{\loc}$ (indeed $C^{2,\,\az}_{\loc}$) in $\Omega$, by Sard's theorem and the strong minimum principle, there exists a sequence $\ez_m\searrow 0$ so that, 
by writing $\Omega_{m}=\{u>\ez_m\}$, one has $\Omega_m$ is $C^2$,
$$\Omega_m\subset \Omega_{m+1}\subset \Omega, $$
and $\Omega_m\to \Omega$ in the Hausdorff distance. Now we are ready to apply the argument of Weinberger by first showing the following identity. 

\begin{lem}\label{Pohozaev 1}
    Let $\Omega$ and $u$ as in Theorem~\ref{main thm}. Then 
    $$n\int_{\Omega}u\,dx = \frac{\mathbf{c}^2 n}2 |\Omega| + \frac{n-2} 2\int_{\Omega} |Du|^2\, dx. $$
\end{lem}
\begin{proof}
As $u\in C^2_\loc$ and $u=\ez_m$ on the boundary of $\partial\Omega_m$, the divergence theorem yields
\begin{align*}
n\int_{\Omega_m} (u-\ez_m)_+\, dx=\ & \int_{\Omega_m} \left[-{\rm div}(x (u-\ez_m)_+) + n (u-\ez_m)_+\right]\, dx\\
=&\ -\int_{\Omega_m} x\cdot Du\, dx\\
=&\ \int_{\Omega_m} \Delta u  (x\cdot Du)\,dx\\
= & \  \int_{\Omega_m} \left[{\rm div}(Du (x\cdot Du))  -   Du\cdot D(x\cdot Du)\right]\, dx\\
= & \  \int_{\Omega_m} {\rm div}(Du (x\cdot Du))\,dx - \int_{\Omega_m}  (Du+x^T D^2u )\cdot Du\, dx.
\end{align*}

As the divergence theorem also gives 
$$\int_{\Omega_m} {\rm div}(Du (x\cdot Du))\,dx = \int_{\partial \Omega_m} (Du\cdot \nu) (x\cdot Du)\,d\mathscr H^{n-1},$$
and
\begin{align*}
    \int_{\Omega_m}x^T D^2u Du\, dx=& \ \frac 1 2\int_{\Omega_m} \left[ {\rm div}(x |Du|^2)-n|Du|^2\right]\, dx \\
    = &\ \frac 1 2\int_{\partial\Omega_m} |Du|^2 (x\cdot \nu)\, d\mathscr H^{n-1} -\frac  n 2 \int_{\Omega_m}|Du|^2\, dx,
\end{align*} 
we employ \eqref{H DH} together with 
$$Du=|Du| \nu \quad \text{ on } \ \partial \Omega_m $$
to conclude
\begin{align*}
    &n\int_{\Omega_m} (u-\ez_m)_+\, dx \\
    &= \ \int_{\partial \Omega_m} (Du\cdot \nu) (x\cdot Du))\,d\mathscr H^{n-1} -  \int_{\Omega_m} |Du|^2\, dx -\int_{\Omega_m} x^T D^2u   Du  \, dx\\
    &= \ \frac 1 2\int_{\partial \Omega_m}  |Du|^2 (x\cdot \nu)\,d\mathscr H^{n-1} +\frac{n-2} 2 \int_{\Omega_m} |Du|^2\, dx.
\end{align*}

Now by letting $m\to \infty$ and applying Lemma~\ref{convergence Laplacian}, we eventually obtain that
$$n\int_{\Omega}u \, dx = \frac{\mathbf{c}^2}{2} \int_{\partial \Omega}   x\cdot \nu\,d\mathscr H^{n-1} +\frac{n-2} 2 \int_{\Omega} |Du|^2\, dx=  \frac{\mathbf{c}^2 n}{2}|\Omega| + \frac{n-2} 2\int_{\Omega} |Du|^2\, dx.$$
\end{proof}

\begin{proof}[Proof of Theorem~\ref{main thm}]
    First of all, taking $\varphi=u$ in \eqref{weak}, we conclude that
$$\int_{\Omega} u\, dx =    \int_{\Omega}  |Du|^2 \, dx.$$
This combined with Lemma~\ref{Pohozaev 1} yields
\begin{equation}\label{identity 1 1} 
\mathbf{c}^2 n|\Omega|=(n+2)\int_{\Omega} u\, dx =  (n+2) \int_{\Omega} |Du|^2\, dx.   
\end{equation}
Particularly, we have
\begin{equation}\label{identity 2 1}
  \int_{\Omega} |Du|^2 \, dx +\frac{2}{n} \int_{\Omega} u\,dx = \mathbf{c}^2|\Omega|.     
\end{equation}

Following the strategy of \cite[Proposition 3.3]{FZ2025},  we next show that
\begin{equation}\label{Du e 1}
u_e(x)=Du(x)\cdot e\le \mathbf c
\end{equation}
for every $x\in \Omega$ and any $e\in \mathbb S^{n-1}$. Towards this, we first note that \eqref{weak} implies
$$\Delta w=0 \quad  \text{ in } \ \Omega, \  \text{where $w=u_e=Du\cdot e$}.$$
Then by fixing $x\in \Omega$ and setting $G_{m}(x,\,y)=G_{m,\,x}(y)$ as the Green's function of Laplacian in $\Omega_m$ containing   $x$, i.e.,
$$\Delta G_{m,\,{x}}=-\delta_{x} + \mu_m\mathscr H^{n-1}|_{\partial\Omega_m},$$
where $\mu_m=- DG_{m,\,x}\cdot \nu_{\Omega_{m}}$ represents the Radon--Nikodym derivative of the harmonic measure on $\partial \Omega_m$,
we have
\begin{equation}\label{ue 1}
    u_e(x)=\int_{\partial\Omega_m}   u_e (y)\mu_m\,d\mathscr H^{n-1}.
\end{equation}

Observe that, for $r>0$ sufficiently small so that $B_r(x)\subset \subset \Omega_m$, we have $(u -\ez_m)_+\ge 0$ in $\Omega$ with
$$G_{m,\,x}\le M(u-\ez_m)_+ \quad \text{ on } \ \partial(\Omega_m\setminus B_r(x))$$
with $M=M(\Omega,x,\,r)>0$. Then as $(u-\ez_m)_+$ solves the same equation as $u$ with  zero boundary value on $\Omega_m$, by the comparison principle, one gets
$$G_{m,\,x}\le M(u-\ez_m)_+ \quad \text{ in } \ \Omega_m\setminus B_r(x).$$
As $\Omega_m$ is $C^2$, this implies
$$|DG_{m,\,x}|\le M|Du| \quad \text{ on } \ \partial \Omega_m. $$
Thus, Lemma~\ref{convergence Laplacian} implies 
$$\mu_m\le M  |Du|\in L^2(\partial\Omega_m;\,\mathscr H^{n-1}). $$ 
This allows us to pass $m\to \infty$ in \eqref{ue 1} as $|u_e|\le |Du|\in L^2(\partial\Omega_m;\,\mathscr H^{n-1})$. Then we conclude \eqref{Du e 1} from facts that $|u_e|\le |Du|\to \mathbf c$ in $L^2$ and that $\mu_md\mathscr H^{n-1}|_{\partial\Omega_m}$ is a sequence of probability measures.
In particular, the arbitrariness of $e$ and $x$ in \eqref{Du e 1} implies
\begin{equation}\label{HDu 1}
    \sup_{\Omega} |Du|\le \mathbf c. 
\end{equation}

We now show that $P(x)=  \frac 1 2|Du|^2(x) +\frac{1}{n}u(x)$ satisfies 
\begin{equation}\label{maximum 1}
    2P(x)\le \mathbf{c}^2. 
\end{equation}
First of all, note that $P$ satisfies
\begin{equation}\label{P 1}
\Delta P = |D^2 u|^2 - \frac 1 n=\biggl|{D^2u}+\frac{1}{n}\textrm{Id}\biggr|^2\ge 0.
    \end{equation}
Thus, $P$ attains its maximum in $\overline{\Omega}_m$ on $\partial\Omega_m$ for every $\ez_m>0$. 
As $\Omega$ is Lipschitz, $u$ is continuous up to the boundary, and  then \eqref{HDu 1} together with the maximum principle above yields \eqref{maximum 1}.

Now we deduce from \eqref{identity 1 1}, \eqref{identity 2 1} and \eqref{maximum 1} that
\begin{align*}
   \mathbf{c}^2|\Omega|   = \frac{n+2}{n} \int_{\Omega} u\,dx =   \int_{\Omega} u\left(-\Delta u+\frac{2}{n}\right)\,dx=\int_{\Omega} \left(|Du|^2+\frac{2}{n} u \right)\,dx\le \mathbf{c}^2|\Omega|,
\end{align*}
which implies that $P$ has to be constant. According to the AM-GM inequality applied in \eqref{P 1}, this yields that   
$D^2u=\frac 1 n Id$.  
Thus the connectedness of $\Omega$ implies
$$u(x)=\frac{r^2-|x|^2}{2n},$$
where  $r>0$ is a constant. This concludes the theorem. 
\end{proof}

\section{Preliminaries for the anisotropic case}

We record some properties of $H$, the Wulff potential. 
By the equivalence of norms in finite dimensional spaces, it follows that
\begin{equation}\label{equiv norm}
C^{-1}|x|\le H(x)\le C|x|,
\end{equation}
for some $C=C(K).$
Moreover, the $1$-homogeneity of $H$ implies
\begin{equation}\label{H DH}
    x\cdot DH(x)=H(x),\quad x\in \mathbb R^n, 
\end{equation}
where the left-hand side is taken as zero when $x=0$. 
In addition,  since $u\in C^2_{\loc}$ and $H\in C^{2}(\mathbb R^n\setminus\{0\})$, then for any $x\in \mathbb R^n\setminus\{0\}$, 
$$D^2V(x)=H(x)D^2H(x) +DH(x)\otimes DH(x)  \quad \text{ is in  $L^\infty(\mathbb R^n)$ and zero-homogeneous},$$ 
and set $H(Du)DH(Du)=0$ whenever $Du=0$. Namely $DV$ is Lipschitz. 
Therefore
$$\Delta_Hu(x)= {\rm div} ( DV(Du(x)))= {\rm div}\left( \int_0^1 D^2V(sDu(x)) Du(x)\, ds    \right):={\rm div}(\mathcal A(x) Du(x)), $$
 is uniformly elliptic with
 $$\mathcal A(x)=\int_0^1 D^2V(sDu(x))  \, ds= D^2V(Du(x)),$$
 since $D^2V$ is $0$-homogeneous and $H$ is uniformly convex. Recall the following terminology firstly introduced in \cite{KP2001}. 
 
\begin{defn}\label{DKP defn}
A real matrix function $\mathcal A\colon \Omega\to \mathbb R^{n\times n}$ satisfies the  $C$-Dahlberg-Kenig-Pipher ($C$-DKP for short) condition in $\Omega$  if $\mathcal A$ satisfies that
$$d\mu_\mathcal A(x)=\chi_{\Omega}(x) \sup\{
d(z)|D\mathcal A(z)|^2 \colon z\in B_{d(x)/4}(x)\}d\mathcal L(x),$$
where $d(z)=\dist(z,\,\partial \Omega)$, 
is a $C$-Carleson measure with respect to $\mathscr H^{n-1}|_{\partial\Omega}$, i.e.,  for any ball $B_r(y)$ centered at $y\in \partial\Omega$, 
$$\mu_\mathcal A(B_r(y)\cap \Omega)\le C \mathscr H^{n-1}(\partial\Omega\cap B_r(y)). $$
\end{defn}

\begin{lem}\label{A DKP}
    Let $\Omega$ and $u$ be  as in Theorem~\ref{main thm 2}. Then 
    $$\mathcal A(x)= D^2V( Du(x))$$
    satisfies the $C\omega_0$-DKP condition in $\Omega$, i.e.
    $$\mu_\mathcal A(B_r(y)\cap \Omega)\le C \| D^2 u\|^{2}_{L^n(\mathcal N_\rho(\partial \Omega)\cap \Omega)}r^{n-1},\quad 0<r<\rho$$
    where, for some $\rho>0$, 
$$\mathcal N_\rho(\partial \Omega): =\{z\in \mathbb R^n\colon \dist(z,\,\partial\Omega)<\rho\}, \quad \omega_0=\|D^2u\|^2_{L^n(\mathcal N_\rho(\partial \Omega)\cap \Omega)},$$
    and
    $$0<C=C(n,\,c_0,\,L,\,\kappa,\,   \|D^3V\|_{L^\infty(\mathbb S^{n-1})}).$$
\end{lem}
\begin{proof}
    Let $y\in \partial\Omega$ and $x\in B_{r}(y)$.
    By differentiating the equation in \eqref{weak wulff u} in the $e_k$-direction, $1\le k\le n$, and letting $v=\partial_k u$, we obtain that $v$ satisfies, in the sense of distribution, 
    $${\rm div} (D^2V(Du) Dv)={\rm div}([H(Du)D^2H(Du)+DH(Du)\otimes DH(Du)]Dv )=0.$$
    According to the De Giorgi--Moser--Nash theory, this implies that $v\in C_\loc^{\alpha}$ with
    $$\alpha=\az(n,\,\kappa,\,\|K\|_{C^1})>0,$$
    where $\kappa$ is the curvature bound of $\partial K$ in Theorem~\ref{main thm 2}. 
    Then by summing over all $k$'s, as $|Du|\neq 0$ in a neighborhood of $\partial \Omega,$ it follows from the Schauder estimates in divergence form together with Poincar\'e inequality that, for any  $0<t\le \frac{d(x)} 4$ with $B_t(x)$ disjoint from $\Omega_0\subset\subset\Omega$, 
    \begin{align}
        \|D^2u\|_{L^\infty(B_t(x))}\le &\ C(n,\,\az,\,\|D^2V\|_{C^\az(\mathbb S^{n-1})})t^{-1} \inf_{{\bf a}\in\mathbb R^n}\bint_{B_{2t}(x)}|Du-{\bf a}|\,d\mathcal L^n\nonumber\\
        \le & \ C(n,\,\az,\,\|D^2V\|_{C^\az(\mathbb S^{n-1})}) \left(\bint_{B_{2t}(x)} |D^2 u|^n\, dx\right)^{\frac 2 n}. \label{bounded hessian}
    \end{align} 

    A direct computation yields
    $$D\mathcal A(z)= D^3V(Du(z))D^2u(z), $$
    and the $(-1)$-homogeneity of $D^3V$ further gives that,  for any $z\in B_{ {d(x)} /4}(x)$, 
    $$|D\mathcal A(z)|= \left|D^3V\left(\frac{Du}{|Du|}(z)\right)\frac{D^2u}{|Du|}(z)\right|\le c_0^{-1}\|D^3V\|_{L^\infty(\mathbb S^{n-1})}\|D^2u\|_{L^\infty(B_{ {d(x)} /4}(x))}.$$
    Therefore, by the Bescovitch covering theorem, we can find a collection of countably many balls 
    $$\{B_{d(x_i)/4}(x_i)\}_{x_i\in B_r(y)\cap \Omega},$$
    with uniformly bounded overlaps up to $C(n)$, covering $B_r(y)\cap \Omega$ so that,  via applying \eqref{bounded hessian} to $u$ on each $B_{d(x_i)/4}(x_i)$, 
\begin{align}
    & \mu_\mathcal A({B_r(y)\cap \Omega})\nonumber\\
    \le& \ C(n)\sum_{i} \mu_\mathcal A(B_{d(x_i)/4}(x_i))\nonumber\\
    \le & \ C(n)\sum_{i} |B_{d(x_i)/4}(x_i)| d(x_i) \|D^3V\|^2_{L^\infty(\mathbb S^{n-1})}\|D^2u\|^2_{L^\infty(B_{d(x_i)/4}(x_i))}\nonumber\\
    \le & \ C(n,\,c_0,\,\az,\, \|D^3V\|_{L^\infty(\mathbb S^{n-1})}) \sum_i  d(x_i)^{n-1} \left(\int_{B_{d(x_i)/2}(x_i)} |D^2 u|^n\, dx\right)^{\frac 2 n}.\label{mu estimate}
\end{align}
Then when $n=2$, via $d(x_i)\le r$ we directly get 
$$ \mu_\mathcal A({B_r(y)\cap \Omega})\le  C(c_0,\,\az,\,\|D^2V\|_{C^\az(\mathbb S^{1})},\,\|D^3V\|_{L^\infty(\mathbb S^{1})}) \| D^2 u\|^{2}_{L^2(\mathcal N_\rho(\partial \Omega)\cap \Omega)}r$$
as desired. 

When $n\ge 3$, we continue from \eqref{mu estimate} via H\"older's inequality that, for $\beta=\frac{n(n-1)}{n-2},$
\begin{align*}
    & \mu_\mathcal A({B_r(y)\cap \Omega})\\
    \le & \  C(n,\,c_0,\,L,\,\az,\, \|D^3V\|_{L^\infty(\mathbb S^{n-1})}) \left( \sum_i \int_{B_{d(x_i)/2}(x_i)} |D^2 u|^n\, dx\right)^{\frac 2 n}\left(\sum_id(x_i)^\beta\right)^{1-\frac 2 n}   \\
    \le & \  C(n,\,c_0,\,L,\,\az,\, \|D^3V\|_{L^\infty(\mathbb S^{n-1})}) \| D^2 u\|^{2}_{L^n(\mathcal N_\rho(\partial \Omega)\cap \Omega)}\left(\sum_id(x_i)^\beta\right)^{1-\frac 2 n}.
\end{align*}
Since there are at most $C(n,\,L)2^{l(n-1)} $-many balls in the collection with radius between $2^{-l} r$ and $2^{-l+1}r$ for each $l\in \mathbb N$, then
$$\left(\sum_id(x_i)^\beta\right)^{ 1-\frac 2 n}\le \left(\sum_l 2^{l(n-1)} 2^{-l(n-1)\frac{n}{n-2}}\right)^{ 1-\frac 2 n} r^{n-1}\le C(n)r^{n-1},$$
where  $-\frac{n}{n-2}+1<0$ was used. 
Thus, we obtain the desired estimate for $y\in \partial \Omega. $   
\end{proof}

\begin{rem}
By invoking the Carleson condition introduced in \cite{DPR2017}, one can slightly weaken some of the assumptions in  such as the $C^3$ regularity of $K$. Nevertheless, the aim of the present manuscript is not to identify optimal conditions on $K$, but rather to demonstrate that our method, originally developed in the isotropic setting, can be extended to anisotropic cases. 
\end{rem}

The following result is a corollary of \cite[Theorem 2.10]{DPR2017}. 

\begin{prop}\label{KP}
    Let $\Omega$ and $u$  be  as in Theorem~\ref{main thm 2}. Then there exists $\delta=\delta(n)>0$ for which,  the nontangential maximal function 
    $N_*(|Du|)\in L^{2+\delta }(\partial\Omega).$
    In particular, 
    $$H(Du)\bigg|_{\partial \Omega_\ez} \to \mathbf{c}$$
      strongly in $L^{2+\delta}$ with respect to $\mathscr H^{n-1}$-measure as $\ez\to 0$, where $\Omega_\ez=\{u>\ez\}$. 
\end{prop}
\begin{proof}
Note that  the assumption of  Theorem~\ref{main thm 2} yields the existence of $\rho>0$ for which
$$D^2u\in L^n(\mathcal N_\rho(\partial\Omega)\cap \Omega) \quad \text{  and } \quad  |Du|(z)\ge c_0>0 \ \text{ for $z\in \mathcal N_\rho(\partial\Omega)\cap \Omega$}, $$
where 
$$\mathcal N_\rho(\partial \Omega): =\{z\in \mathbb R^n\colon \dist(z,\,\partial\Omega)<\rho\}.$$ 
Then $Du\in VMO(\mathcal N_\rho(\partial\Omega)\cap \Omega)$ directly by Sobolev embedding theorem. 

Since $\Omega$ is Lipschitz, we can extend 
$$\mathcal A(x)=D^2V(Du)\in VMO(\mathcal N_\delta(\partial \Omega)\cap \Omega), \quad D\mathcal A(z)= D^3V(Du(z))D^2u(z)\in L^n(\mathcal N_\rho(\partial\Omega)\cap \Omega)$$ 
to $\mathbb R^n$ with $D\mathcal A(z)\in L^n(\mathcal N_\rho(\partial\Omega))$ and satisfying the ellipticity condition, and still denote the extended one by $\mathcal A$.  

Now we take a large ball $B_R$ with $\Omega\subset\subset B_R$. Let $v$ solves
$${\rm div}( {\mathcal A}(x)Dv)=-\chi_{\Omega} \ \text{ in } B_R, \quad  v=0 \ \text{ on } \partial B_R. $$
Then since $D^2u\in L^n(\mathcal N_\rho(\partial\Omega)\cap \Omega), $
$${\rm div}( {\mathcal A}(x)Dv)= \mathcal A D^2v +  \partial_{i} \mathcal A_{ij} \partial_j v$$
is uniformly elliptic, and  $v\in W^{2,\,2}(\mathcal N_\rho(\partial\Omega) ) $ by the Calder\'on-Zygmund theory on VMO-coefficients; see e.g. \cite{C2004}.

Thus, $u-v$ is  $\mathcal A$-harmonic in $\Omega$ and, since $u\in W^{2,\,2}(\mathcal N_\rho(\partial\Omega)\cap \Omega)$,  
$$(u-v)|_{\partial\Omega}=-v|_{\partial\Omega}$$ 
is the trace of a $W^{2,\,2}$-function on the boundary $\partial\Omega$ of a Lipschitz domain. Thus employing fractional Sobolev embeddings, we have 
$$D(u-v)|_{\partial\Omega}\in W^{\frac 1 2, 2}(\partial\Omega;\,d\mathscr H^{n-1})\hookrightarrow L^{2+\delta}(\partial\Omega;\,d\mathscr H^{n-1}).$$
Now according to Lemma~\ref{A DKP}, by taking $\rho>0$ small enough so that $$ \|D^2u\|^2_{L^n(\mathcal N_\rho(\partial \Omega)\cap \Omega)}\le \ez_0\ll1$$ 
we can apply  \cite[Theorem 2.10]{DPR2017}  to $v-u$ to conclude the proposition. 
\end{proof}

Now we can derive Theorem~\ref{main thm 2} through a straightforward   adaptation of the proof for Theorem~\ref{main thm}. We give the detailed steps  in the Appendix for reference.

\appendix
\section{Proof of Theorem~\ref{main thm 2}}

Following the idea of  Weinberger \cite{W1971}, we prove Theorem~\ref{main thm 2} in terms of the nontangential limit.  

Recall that, there exists a sequence $\ez_m\searrow 0$ so that, 
by writing $\Omega_{m}=\{u>\ez_m\}$, one has $\Omega_m$ is $C^2$,
$$\Omega_m\subset \Omega_{m+1}\subset \Omega, $$
and $\Omega_m\to \Omega$ in the Hausdorff distance. 
Moreover, $V(x)=\frac 1 2 H^2(x)$ has a quadratic growth by \eqref{equiv norm}. We firstly show the following version of Pohozaev identity. 
\begin{lem}\label{Pohozaev}
    Let $\Omega$ and $u$ as in Theorem~\ref{main thm}. Then 
    $$n\int_{\Omega}u\,dx = \frac{\mathbf{c}^2 n}2 |\Omega| + (n-2)\int_{\Omega} V(Du)\, dx. $$
\end{lem}
\begin{proof}
As $u\in C^2_\loc$ and $u=\ez_m$ on the boundary of $\partial\Omega_m$, the divergence theorem yields
\begin{align*}
n\int_{\Omega_m} (u-\ez_m)_+\, dx=\ & \int_{\Omega_m} \left[-{\rm div}(x (u-\ez_m)_+) + n (u-\ez_m)_+\right]\, dx\\
=&\ -\int_{\Omega_m} x\cdot Du\, dx\\
=&\ \int_{\Omega_m} {\rm div}(DV(Du)) (x\cdot Du)\,dx\\
= & \  \int_{\Omega_m} \left[{\rm div}(DV(Du) (x\cdot Du))  -   DV(Du)\cdot D(x\cdot Du)\right]\, dx\\
= & \  \int_{\Omega_m} {\rm div}(DV(Du) (x\cdot Du))\,dx - \int_{\Omega_m}  (Du+x^T D^2u )\cdot DV(Du)\, dx.
\end{align*}

As the divergence theorem also gives 
$$\int_{\Omega_m} {\rm div}(DV(Du) (x\cdot Du))\,dx = \int_{\partial \Omega_m} (DV(Du)\cdot \nu) (x\cdot Du)\,d\mathscr H^{n-1},$$
and
\begin{align*}
    \int_{\Omega_m}x^T D^2u DV(Du)\, dx=& \ \int_{\Omega_m} \left[ {\rm div}(x V(Du))-nV(Du)\right]\, dx \\
    = &\ \int_{\partial\Omega_m} V(Du) (x\cdot \nu)\, d\mathscr H^{n-1} - n \int_{\Omega_m}V(Du)\, dx,
\end{align*} 
we employ \eqref{H DH} together with 
$$Du=|Du| \nu \quad \text{ on } \ \partial \Omega_m $$
to conclude
\begin{align*}
    &n\int_{\Omega_m} (u-\ez_m)_+\, dx \\
    &= \ \int_{\partial \Omega_m} (DV(Du)\cdot \nu) (x\cdot Du))\,d\mathscr H^{n-1} - 2 \int_{\Omega_m} V(Du)\, dx -\int_{\Omega_m} x^T D^2u DV(Du) \, dx\\
    &= \ \int_{\partial \Omega_m}  V(Du) (x\cdot \nu)\,d\mathscr H^{n-1} +(n-2) \int_{\Omega_m} V(Du)\, dx.
\end{align*}

Now by letting $m\to \infty$ and applying Proposition~\ref{KP}, we eventually obtain that
$$n\int_{\Omega}u \, dx = \frac{\mathbf{c}^2}{2} \int_{\partial \Omega}   x\cdot \nu\,d\mathscr H^{n-1} +(n-2) \int_{\Omega} V(Du)\, dx=  \frac{\mathbf{c}^2 n}{2}|\Omega| + (n-2) \int_{\Omega} V(Du)\, dx,$$
recalling that $V=\frac 1 2 H^2$ has a quadratic growth by \eqref{equiv norm}  and hence $V(Du)\to \frac {\mathbf c^2} 2$.
\end{proof}

Now we are ready to show Theorem~\ref{main thm 2}. 
\begin{proof}[Proof of Theorem~\ref{main thm 2}]
First of all, testing \eqref{weak wulff u} via $u$, we conclude from \eqref{H DH} that
$$\int_{\Omega} u\, dx = \int_{\Omega} Du\cdot DV(Du)\, dx= 2 \int_{\Omega} V(Du)\, dx.$$
This combined with Lemma~\ref{Pohozaev} yields
\begin{equation}\label{identity 1} 
\mathbf{c}^2 n|\Omega|=(n+2)\int_{\Omega} u\, dx = 2(n+2) \int_{\Omega} V(Du)\, dx.   
\end{equation}
Particularly, we have
\begin{equation}\label{identity 2}
 2\int_{\Omega} V(Du) \, dx +\frac{2}{n} \int_{\Omega} u\,dx = \mathbf{c}^2|\Omega|.     
\end{equation}

Following again the strategy of \cite[Proposition 3.3]{FZ2025},  we next show that
\begin{equation}\label{Du e}
u_e(x)=Du(x)\cdot e\le \mathbf c
\end{equation}
for every $x\in \Omega$ and any $e\in \partial K$. Towards this, we first note that \eqref{weak wulff u} implies
$${\rm div}(D^2V(Du) Du_e)=:L_\mathcal A(w)=0 \quad  \text{ in } \ \Omega, \  \text{where $w=Du\cdot e$}.$$
Then by fixing $x\in \Omega$ and setting $G_{m}(x,\,y)=G_{m,\,x}(y)$ as the Green's function of the operator $L_\mathcal A$ in $\Omega_m$ containing   $x$, i.e.,
$${\rm div}(D^2V(Du) DG_{m,\,{x}})=-\delta_{x} + \mu_m\mathscr H^{n-1}|_{\partial\Omega_m},$$
where $\mu_m=-D^2V(Du)DG_{m,\,x}\cdot \nu_{\Omega_{m}}$ represents the Radon--Nikodym derivative of the $\mathcal A$-harmonic measure on $\partial \Omega_m$,
we have
\begin{equation}\label{ue}
    u_e(x)=\int_{\partial\Omega_m}   u_e (y)\mu_m\,d\mathscr H^{n-1}.
\end{equation}

Observe that, for $r>0$ sufficiently small so that $B_r(x)\subset \subset \Omega_m$, we have $(u -\ez_m)_+\ge 0$ in $\Omega$ with
$$G_{m,\,x}\le M(u-\ez_m)_+ \quad \text{ on } \ \partial(\Omega_m\setminus B_r(x))$$
with $M=M(\Omega,x,\,r)>0$. Then as $(u-\ez_m)_+$ solves the same equation as $u$ with  zero boundary value on $\Omega_m$, by the comparison principle, one gets
$$G_{m,\,x}\le M(u-\ez_m)_+ \quad \text{ in } \ \Omega_m\setminus B_r(x).$$
As $\Omega_m$ is $C^2$, we get
$$|DG_{m,\,x}|\le M|Du| \quad \text{ on } \ \partial \Omega_m. $$
Thus, Proposition~\ref{KP} together with the $0$-homogeneity of $D^2v$ implies 
$$\mu_m\le M\|D^2V\|_{L^\infty(\mathbb S^{n-1})} |Du|\in L^2(\partial\Omega_m;\,\mathscr H^{n-1}). $$ 
This allows us to pass $m\to \infty$ in \eqref{ue} as $|u_e|\le H(Du)\in L^2(\partial\Omega_m;\,\mathscr H^{n-1})$. Then we conclude \eqref{Du e} from facts that $|u_e|\le H(Du)\to \mathbf c$ in $L^2$ and that $\mu_md\mathscr H^{n-1}|_{\partial\Omega_m}$ is a sequence of probability measures.
In particular, the arbitrariness of $e$ and $x$ in \eqref{Du e} implies
\begin{equation}\label{HDu}
    \sup_{\Omega} H(Du)\le \mathbf c. 
\end{equation}

We now show that $P(x)=  V(Du(x)) +\frac{1}{n}u(x)$ satisfies 
\begin{equation}\label{maximum}
    2P(x)\le \mathbf{c}^2. 
\end{equation}
First of all, recall from \cite[Theorem 4 \& Formula (31)]{WX2011} that $P$ satisfies
\begin{equation}\label{P}
\begin{aligned}
    &{\rm tr} (D^2V(Du) D^2P) + b_i P_i=\sum_{i,\,j}\lambda_i\lambda_j(\partial_{ij} u)^2 - \frac  1n\\
    &\ge  \sum_{i}\lambda_i^2(\partial_{ii} u)^2 - \frac 1 n \ge \frac 1 n  (\sum_i\lambda_i\partial_{ii}u)^2- \frac 1 n=0 
\end{aligned}
    \end{equation}
where $\mathbf{b}= (b_i)_{i=1}^n$ is a vector coefficient and $\lambda_i>0, 1\le i\le n,$ are the eigenvalues of $D^2V(Du)$; note that  in a suitable coordinate system, 
$$\Delta _Hu = \sum_i\lambda_i\partial_{ii}u=-1. $$ 
Thus, $P$ attains its maximum in $\overline{\Omega}_m$ on $\partial\Omega_m$ for every $\ez_m>0$. 
As $\Omega$ is Lipschitz, $u$ is continuous up to the boundary, and  then \eqref{HDu} together with the maximum principle above yields \eqref{maximum}.  


Now we deduce from \eqref{identity 1}, \eqref{identity 2} and \eqref{maximum} that
\begin{align*}
   \mathbf{c}^2|\Omega|   = \frac{n+2}{n} \int_{\Omega} u\,dx =   \int_{\Omega} u\left(-{\rm div}(DV(Du))+\frac{2}{n}\right)\,dx=\int_{\Omega} \left(2V(Du)+\frac{2}{n} u \right)\,dx\le \mathbf{c}^2|\Omega|,
\end{align*}
which implies that $P$ has to be constant. According to the AM-GM inequality applied in \eqref{P}, this yields that  $\lambda_i\partial_{ii}u$  must be the same, i.e., 
$$D^2V(Du)D^2u=D(DV(Du)) \quad \text{is a multiple of identity.}$$
Thus the connectedness of $\Omega$ implies
$$u(x)=\frac{r^2-H_*^2(x)}{2n},$$
where  $r>0$ is a constant and  $H_*$ is the dual function of $H$; see e.g.\ \cite[(5.17)-(5.19)]{CS2009}. This concludes the theorem. 
\end{proof}

\noindent{\bf Conflict of interest statement}: All authors declare that they have no conflicts of interest to disclose.


\begin{thebibliography}{99}

\bibitem{A1958} 
A. D. Alexandrov, \emph{Uniqueness theorem for surfaces in the large}. V, Vestnik, Leningrad Univ. 13, 19 (1958), 5--8, Amer. Math. Soc. Transl. 21, Ser. 2, 412--416.

\bibitem{A1962} 
A. D. Alexandrov, \emph{A characteristic property of spheres}. Ann. Math. Pura Appl. 58 (1962), 303--315.

\bibitem{ACMM2001}
 L. Ambrosio, V. Caselles, S. Masnou, J.-M. Morel, \emph{Connected components of sets of finite perimeter and applications to image processing}. J. Eur. Math. Soc. (JEMS) 3 (2001),  no. 1, 39--92.

\bibitem{BNST2008}
B. Brandolini, C. Nitsch, P. Salani, C. Trombetti, \emph{Serrin type overdetermined problems: an alternative proof}. Arch. Ration. Mech. Anal. 190 (2008), no. 2, 267--280.


\bibitem{C1985}
A. P. Calder\'on,  
\emph{Boundary value problems for the Laplace equation in Lipschitzian domains}. Recent progress in Fourier analysis (El Escorial, 1983), 33--48,
North-Holland Math. Stud., 111, Notas Mat., 101, North-Holland, Amsterdam, 1985.

\bibitem{C2004}
 Y. M. Chen, \emph{Regularity of solutions to elliptic equations with VMO coefficients.} 
Acta Math. Sin. (Engl. Ser.) 20 (2004), no. 6, 1103--1118.

\bibitem{CH1998}
 M. Choulli and A. Henrot, \emph{Use of the domain derivative to prove symmetry results in partial differential
 equations}. Math. Nachr. 192 (1998), 91--103.

\bibitem{CS2009}
A. Cianchi, P. Salani, \emph{Overdetermined anisotropic elliptic problems}. Math. Ann. 345 (2009), no. 4, 859--881.

\bibitem{CM2017}
G. Ciraolo, F. Maggi, {\em On the shape of compact hypersurfaces with almost constant mean curvature.} Comm.
 Pure Appl. Math. 70 (2017), 665--716.

\bibitem{DK1987}
B. E. J. Dahlberg, C.  Kenig,  
Hardy spaces and the Neumann problem in $L^p$ for Laplace's equation in Lipschitz domains.
Ann. of Math. (2) 125 (1987), no. 3, 437--465.

\bibitem{DPR2017}
M. Dind\u{o}s, J. Pipher, D. Rule, \emph{Boundary value problems for second-order elliptic operators satisfying a Carleson condition.}
Comm. Pure Appl. Math. 70 (2017), no. 7, 1316--1365

\bibitem{D2024}
A. De Rosa,  \emph{On the theory of anisotropic minimal surfaces}. Notices Amer. Math. Soc. 71 (2024), no. 7, 853--859.


\bibitem{DKS2020}
A. De Rosa, S.  Kolasi\'nski, M. Santilli, \emph{Uniqueness of critical points of the anisotropic isoperimetric problem for finite perimeter sets}. Arch. Ration. Mech. Anal. 238 (2020), no. 3, 1157--1198.

\bibitem{DM2019}
M. G. Delgadino,  F. Maggi, 
\emph{Alexandrov's theorem revisited. }
Anal. PDE 12 (2019), no. 6, 1613--1642.

\bibitem{FZ2025}
A. Figalli, Y. R.-Y. Zhang, \emph{Serrin's overdetermined problem in rough domains}, J. Eur. Math. Soc., To appear.

\bibitem{FZ2026}
A. Figalli, Y. R.-Y. Zhang, \emph{An anisotropic Serrin's problem in rough domains}, In preparation. 

\bibitem{HLMG2009}
Y. He, H. Li, H. Ma, J. Ge, \emph{Compact embedded hypersurfaces with constant higher order anisotropic mean curvatures}. Indiana Univ. Math. J. 58 (2009), no. 2, 853--868.

\bibitem{HLL2024}
 Y. Huang, Q. Li, Q. Li, \emph{Concentration breaking on two optimization problems}. Sci. China Math. (2024).

\bibitem{K1985}
C. Kenig, \emph{Recent progress on boundary value problems on Lipschitz domains}. Pseudodifferential operators and applications (Notre Dame, Ind., 1984), 175--205, Proc. Sympos. Pure Math., 43, Amer. Math. Soc., Providence, RI, 1985.

\bibitem{KP1993}
C. Kenig, J. Pipher, \emph{The Neumann problem for elliptic equations with non-smooth coefficients}.
 Invent. Math., 113(3)  (1993), 447--509.

\bibitem{M2012}
F. Maggi, 
\emph{Sets of finite perimeter and geometric variational problems}. 
An introduction to geometric measure theory. Cambridge Studies in Advanced Mathematics, 135. Cambridge University Press, Cambridge, 2012. 

\bibitem{M2018}
F. Maggi, \emph{Critical and almost-critical points in isoperimetric problems}. Oberwolfach Rep. 35 (2018), 34--37.  

\bibitem{MR1991}
S. Montiel,  A. Ros, \emph{Compact hypersurfaces: the Alexandrov theorem for higher order mean curvatures}. In Differential geometry, volume 52 of Pitman Monogr. Surveys Pure Appl. Math., pp. 279--296. Longman Sci. Tech., Harlow, 1991.

\bibitem{KP2001}
C.  Kenig, J. Pipher, \emph{The Dirichlet problem for elliptic equations with drift terms}. Publ. Mat. 45 (2001), no. 1, 199--217.


\bibitem{M2017}
R. Magnanini,
\emph{Alexandrov, Serrin, Weinberger, Reilly: symmetry and stability by integral identities}. 
Bruno Pini Math. Anal. Semin., 8
Universit\`a di Bologna, Alma Mater Studiorum, Bologna, 2017, 121--141.

\bibitem{MP2019}
R. Magnanini, G. Poggesi, 
\emph{On the stability for Alexandrov's soap bubble theorem. }
J. Anal. Math. 139 (2019), no. 1, 179--205.

\bibitem{MP2020}
R. Magnanini, G. Poggesi, 
\emph{Serrin's problem and Alexandrov's soap bubble theorem: enhanced stability via integral identities.}
Indiana Univ. Math. J. 69 (2020), no. 4, 1181--1205.

\bibitem{PS1989}
L.E. Payne, P.W. Schaefer, \emph{Duality theorems in some overdetermined problems}. Math. Methods Appl. Sci. 11 (1989), 805--819.

\bibitem{P1998}
J. Prajapat, 
\emph{Serrin's result for domains with a corner or cusp.}
Duke Math. J. 91 (1998), no. 1, 29-31.

\bibitem{S1971}
J. Serrin, \emph{A symmetry problem in potential theory}. Arch. Ration. Mech. Anal. 43 (1971), 304--318.

\bibitem{S1993}
R. Schneider, \emph{Convex Bodies: The Brunn-Minkowski Theory}. Encyclopedia of Mathematics and its Applications, 44. Cambridge University Press, Cambridge, 1993.

\bibitem{V2023}
B. Velichkov, \emph{Regularity of the one-phase free boundaries}. Springer Nature, 2023.

\bibitem{V1984}
G. Verchota, 
\emph{Layer potentials and regularity for the Dirichlet problem for Laplace's equation in Lipschitz domains}.
J. Funct. Anal. 59 (1984), no. 3, 572--611.

\bibitem{V1992}
A. L. Vogel, \emph{Symmetry and regularity for general regions having a solution to certain overdetermined boundary value problems}. Atti Sem. Mat. Fis. Univ. Modena 40 (1992), no. 2, 443--484. 

\bibitem{WX2011}
G. Wang, C. Xia, 
\emph{A characterization of the Wulff shape by an overdetermined anisotropic PDE.}
Arch. Ration. Mech. Anal. 199 (2011), no. 1, 99--115.

\bibitem{W1971}
H. F. Weinberger, \emph{Remark on the preceding paper of Serrin}. Arch. Ration. Mech. Anal.
43 (1971), 319--320.

\end{thebibliography}
\end{document}